\documentclass[10pt, reqno]{amsart}

\usepackage{amsfonts}
\usepackage{amsmath}
\usepackage{amssymb}
\usepackage{amscd}
\usepackage[mathscr]{eucal}
\usepackage{url}

\allowdisplaybreaks[1]

\newcommand{\ph}[2]{{\left({#1}\right)}_{#2}}

\renewcommand*{\bar}{\overline}
\newcommand{\gfp}[1]{\Gamma_p{\left({#1}\right)}}
\newcommand{\biggfp}[1]{\Gamma_p{\bigl({#1}\bigr)}}
\newcommand{\Biggfp}[1]{\Gamma_p{\Bigl({#1}\Bigr)}}

\theoremstyle{plain}

\newtheorem{theorem}{Theorem}[section]
\newtheorem{lemma}[theorem]{Lemma}
\newtheorem{prop}[theorem]{Proposition}
\newtheorem{cor}[theorem]{Corollary}
\theoremstyle{definition}

\newtheorem{conj}[theorem]{Conjecture}

\numberwithin{equation}{section}

\begin{document}

\title[On a supercongruence conjecture of Rodriguez-Villegas]{On a supercongruence conjecture of Rodriguez-Villegas}
\author{Dermot M\lowercase{c}Carthy}  

\address{Department of Mathematics, Texas A\&M University, College Station, TX 77843-3368, USA}

\email{mccarthy@math.tamu.edu}




\begin{abstract}
In examining the relationship between the number of points over $\mathbb{F}_p$ on certain Calabi-Yau manifolds and hypergeometric series which correspond to a particular period of the manifold, Rodriguez-Villegas identified numerically 22 possible supercongruences. We prove one of the outstanding supercongruence conjectures between a special value of a truncated generalized hypergeometric series and the $p$-th Fourier coefficient of a modular form. 
\end{abstract}

\maketitle

\section{Introduction and Statement of Results}
\makeatletter{\renewcommand*{\@makefnmark}{}
\footnotetext{This work was supported by the UCD Ad Astra Research Scholarship program.}
Let $\mathbb{F}_{p}$ denote the finite field with $p$, a prime, elements. In \cite{R} Rodriguez-Villegas examined the relationship between the number of points over $\mathbb{F}_p$ on certain Calabi-Yau manifolds and truncated generalized hypergeometric series which correspond to a particular period of the manifold. In doing so, he identified numerically 22 possible supercongruences which can be categorized by the dimension, $D$, of the manifold as outlined below.

We first define the truncated generalized hypergeometric series. For a complex number $a$ and a non-negative integer $n$ let $\ph{a}{n}$ denote the rising factorial defined by
\begin{equation*}\label{RisFact}
\ph{a}{0}:=1 \quad \textup{and} \quad \ph{a}{n} := a(a+1)(a+2)\dotsm(a+n-1) \textup{ for } n>0.
\end{equation*}
Then, for complex numbers $a_i$, $b_j$ and $z$, with none of the $b_j$ being negative integers or zero, we define the truncated generalized hypergeometric series
\begin{equation*}
{{_rF_s} \left[ \begin{array}{ccccc} a_1, & a_2, & a_3, & \dotsc, & a_r \vspace{.05in}\\
\phantom{a_1} & b_1, & b_2, & \dotsc, & b_s \end{array}
\Big| \; z \right]}_{m}
:=\sum^{m}_{n=0}
\frac{\ph{a_1}{n} \ph{a_2}{n} \ph{a_3}{n} \dotsm \ph{a_r}{n}}
{\ph{b_1}{n} \ph{b_2}{n} \dotsm \ph{b_s}{n}}
\; \frac{z^n}{{n!}}.
\end{equation*}
We also let $\phi(\cdot)$ denote Euler's totient function and $\left(\frac{\cdot}{p}\right)$ the Legendre symbol modulo $p$.

For $D=1$, associated to certain elliptic curves, 4 supercongruences were identified. They were all of the form 
\begin{equation*}
{_{2}F_1} \Biggl[ \begin{array}{cc} \frac{1}{d}, & 1-\frac{1}{d}\vspace{.05in}\\
\phantom{\frac{1}{d}} & 1 \end{array}
\bigg| \; 1 \Biggr]_{p-1}
\equiv
\left(\frac{-t}{p}\right)
\pmod{p^2},
\end{equation*}
where $\phi(d) \leq 2$, $1\leq t \leq3$ and $p$ is a prime not dividing $d$. These cases have been proven by Mortenson  \cite{M1}, \cite{M2}.

For $D=2$ another 4 supercongruences were identified which relate to certain modular K3 surfaces. These were all of the form 
\begin{equation*}
{_{3}F_2} \Biggl[ \begin{array}{ccc} \frac{1}{2}, & \frac{1}{d}, & 1-\frac{1}{d}\vspace{.05in}\\
\phantom{\frac{1}{d}} & 1, &1 \end{array}
\bigg| \; 1 \Biggr]_{p-1}
\equiv
a(p)
\pmod{p^2},
\end{equation*}
where $\phi(d) \leq 2$, $p$ is a prime not dividing $d$ and $a(p)$ is the $p$-th Fourier coefficient of a weight three modular form on a congruence subgroup of $SL(2,\mathbb{Z})$. The case when $d=2$ was originally conjectured by Beukers and Stienstra \cite{BS} and was first proven by Van Hamme \cite{VH}. Subsequently, proofs were also provided by Ishikawa \cite{I} and Ahlgren \cite{A}. The other $D=2$ cases are dealt with by Mortenson \cite{M} where they have been proven for $p\equiv 1 \pmod d$ and up to sign otherwise. 

The remaining 14 supercongruence conjectures relate to Calabi-Yau threefolds (i.e. $D=3$). The threefolds in question are complete intersections of hypersurfaces, of which 13 are discussed by Batyrev and van Straten in \cite{BvS}. The supercongruences can be expressed as either
\begin{equation*}
{_{4}F_3} \Biggl[ \begin{array}{cccc} \frac{1}{d_1}, & 1-\frac{1}{d_1}, & \frac{1}{d_2}, & 1-\frac{1}{d_2}\vspace{.05in}\\
\phantom{\frac{1}{d_1}} & 1, & 1, & 1 \end{array}
\bigg| \; 1 \Biggr]_{p-1}
\equiv
b(p)
\pmod{p^3},
\end{equation*}
where $\phi(d_i) \leq 2$ and $p$ is a prime not dividing $d_i$, or
\begin{equation}\label{type2}
{_{4}F_3} \Biggl[ \begin{array}{cccc} \frac{1}{d}, & \frac{r}{d}, & 1-\frac{r}{d}, & 1-\frac{1}{d}\vspace{.05in}\\
\phantom{\frac{1}{d_1}} & 1, & 1, & 1 \end{array}
\bigg| \; 1 \Biggr]_{p-1}
\equiv
b(p)
\pmod{p^3},
\end{equation}
where $\phi(d) =4$, $1<r<d-1$ with $\gcd(r,d)=1$, $p$ is a prime not dividing $d$ and $b(p)$ is the $p$-th Fourier coefficient of a weight four modular form on a congruence subgroup of $SL(2,\mathbb{Z})$. Only one of these cases has been proven (Kilbourn \cite{K}). It is of the first type with $d_1=d_2=2$ and is an extension of the Ap\'ery number supercongruence \cite{AO}.

Let 
\begin{equation}\label{for_ModForm}
f(z):= f_1(z)+5f_2(z)+20f_3(z)+25f_4(z)+25f_5(z)=\sum_{n=1}^{\infty} c(n) q^n
\end{equation}
where $f_i(z):=\eta^{5-i}(z) \hspace{2pt} \eta^4(5z) \hspace{2pt} \eta^{i-1}(25z)$,
$\eta(z):=q^{\frac{1}{24}} \prod_{n=1}^{\infty}(1-q^n)$ is the Dedekind eta function and $q:=e^{2 \pi i z}$. Then $f$ is a cusp form of weight four on the congruence subgroup $\Gamma_0(25)$.  We now list one of the outstanding conjectures of type (\ref{type2}).
\begin{conj}[Rodriguez-Villegas \cite{R}]\label{conj_RV} If $p \neq 5$ is prime and $c(p)$ is as defined in (\ref{for_ModForm}), then 
\begin{equation*}
{_4F_3} \Biggl[ \begin{array}{cccc} \frac{1}{5}, & \frac{2}{5}, & \frac{3}{5}, & \frac{4}{5} \\
\phantom{\frac{1}{5},} & 1, & 1, & 1\end{array}
\bigg| \; 1 \Biggr]_{p-1}
\equiv c(p) \pmod {p^3}.
\end{equation*}
\end{conj}
\noindent The main result of this paper is the following theorem.
\begin{theorem}\label{thm_DMCMain} Conjecture \ref{conj_RV} is true.
\end{theorem}

The main approach for proving these types of supercongruences has been to use congruence relations between truncated generalized hypergeometric series and Greene's hypergeometric functions over finite fields \cite{A, AO, K, MO, M, M1, M2}. However, as noted in \cite{DMC1}, many results using this approach are restricted to primes in certain congruence classes (e.g. $p\equiv 1 \pmod d$ in some of the $D=2$ cases described above). In \cite{DMC1}, this author extends Greene's hypergeometric functions to the $p$-adic setting and establishes congruences between this new function and certain truncated generalized hypergeometric series. These congruences cover all 22 hypergeometric series outlined above and are valid for all primes required in each of these cases, thus providing a framework for proving all 22 cases. The proof of Theorem \ref{thm_DMCMain} relies on one of these congruences along with counting the number of rational points on a modular Calabi-Yau threefold over $\mathbb{F}_p$. 

Specifically, let $p$ be an odd prime and let $n \in \mathbb{Z}^{+}$. For $1 \leq i \leq n+1$, let $\frac{m_i}{d_i} \in \mathbb{Q} \cap \mathbb{Z}_p$ such that $0<\frac{m_i}{d_i}<1$. 
Let $\gfp{\cdot}$ denote Morita's $p$-adic gamma function, $\left\lfloor x \right\rfloor$ denote the greatest integer less than or equal to $x$ and
$\langle x \rangle$ denote the fractional part of $x$, i.e. $x- \left\lfloor x \right\rfloor$. Then define
\begin{multline}\label{for_GFn}
{_{n+1}G} \left( \tfrac{m_1}{d_1}, \tfrac{m_2}{d_2}, \dotsc, \tfrac{m_{n+1}}{d_{n+1}} \right)_p\\
:= \frac{-1}{p-1}  \sum_{j=0}^{p-2} 
{\left((-1)^j\biggfp{\tfrac{j}{p-1}}\right)}^{n+1} 
\prod_{i=1}^{n+1} \frac{\biggfp{\langle \frac{m_i}{d_i}-\frac{j}{p-1}\rangle}}{\biggfp{\frac{m_i}{d_i}}}
(-p)^{-\lfloor{\frac{m_i}{d_i}-\frac{j}{p-1}\rfloor}}.
\end{multline}

\noindent An example of one the supercongruence results from \cite{DMC1} is the following theorem.
\begin{theorem}[\cite{DMC1} Thm. 2.7]\label{thm_4G2}
Let $r, d \in \mathbb{Z}$ such that $2 \leq r \leq d-2$ and $\gcd(r,d)=1$. Let $p$ be an odd prime such that $p\equiv \pm1 \pmod d$ or $p\equiv \pm r \pmod d$ with $r^2 \equiv \pm1 \pmod d$. If $s(p) := \gfp{\tfrac{1}{d}} \gfp{\tfrac{r}{d}}\gfp{\tfrac{d-r}{d}}\gfp{\tfrac{d-1}{d}}$, then
\begin{multline*}
{_{4}G} \Bigl(\tfrac{1}{d} , \tfrac{r}{d}, 1-\tfrac{r}{d} , 1-\tfrac{1}{d}\Bigr)_p\\
\equiv
{_{4}F_3} \Biggl[ \begin{array}{cccc} \frac{1}{d}, & \frac{r}{d}, & 1-\frac{r}{d}, & 1-\frac{1}{d}\vspace{.05in}\\
\phantom{\frac{1}{d_1}} & 1, & 1, & 1 \end{array}
\bigg| \; 1 \Biggr]_{p-1}
+s(p)\hspace{1pt} p
\pmod {p^3}.
\end{multline*}
\end{theorem}

\noindent Taking $d=5$ in Theorem \ref{thm_4G2} yields
\begin{equation*}
{_{4}G} \left(\tfrac{1}{5} ,\hspace{2pt} \tfrac{2}{5},\hspace{2pt} \tfrac{3}{5} ,\hspace{2pt} \tfrac{4}{5}\right)_p
-s(p)\hspace{1pt} p
\equiv
{_{4}F_3} \Biggl[ \begin{array}{cccc} \frac{1}{5}, & \frac{2}{5}, & \frac{3}{5}, & \frac{4}{5}\vspace{.05in}\\
\phantom{\frac{1}{d_1}} & 1, & 1, & 1 \end{array}
\bigg| \; 1 \Biggr]_{p-1}
\pmod {p^3}.
\end{equation*}
Therefore Theorem \ref{thm_DMCMain} will be established (after checking the case when $p=2$) on proving the following theorem.
\begin{theorem}\label{cor_GStoMod}
If $p \neq 5$ is an odd prime, $s(p):= \gfp{\frac{1}{5}} \gfp{\frac{2}{5}} \gfp{\frac{3}{5}} \gfp{\frac{4}{5}}$ and $c(p)$ is as defined in (\ref{for_ModForm}), then 
\begin{align*}
{_{4}G} \left(\tfrac{1}{5} ,\hspace{2pt} \tfrac{2}{5},\hspace{2pt} \tfrac{3}{5} ,\hspace{2pt} \tfrac{4}{5}\right)_p
-s(p)\hspace{1pt} p
&=
c(p).
\end{align*}
\end{theorem}

As mentioned above, ${_{n+1}G}$ extends Greene's hypergeometric function over finite fields, which was introduced in \cite{G}. Let $\widehat{\mathbb{F}_p^{*}}$ denote the group of multiplicative characters of $\mathbb{F}_p^{*}$. 
We extend the domain of $\chi \in \widehat{\mathbb{F}_p^{*}}$ to $\mathbb{F}_{p}$ by defining $\chi(0):=0$ (including the trivial character $\varepsilon$) and denote $\bar{B}$ as the inverse of $B$. For $A$, $B \in \widehat{\mathbb{F}_p^{*}}$, define
$
\binom{A}{B} := 
\frac{B(-1)}{p} \sum_{x \in \mathbb{F}_{p}} A(x) \bar{B}(1-x).
$
Then for $A_0,A_1,\dotsc, A_n, B_1, B_2, \dotsc, B_n \in \widehat{\mathbb{F}_p^{*}}$ and $x \in \mathbb{F}_{p}$, define the Gaussian hypergeometric series by
\begin{equation*}
{_{n+1}F_n} {\Biggl( \begin{array}{cccc} A_0, & A_1, & \dotsc, & A_n \\
\phantom{A_0} & B_1, & \dotsc, & B_n \end{array}
\bigg| \; x \Biggr)}_{p}
:= \frac{p}{p-1} \sum_{\chi} \binom{A_0 \chi}{\chi} \prod_{i=1}^{n} \binom{A_i \chi}{B_i \chi}
\, \chi(x).
\end{equation*}

\noindent We recover the Gaussian hypergeometric series from $_{n+1}G$ via the following result. 
\begin{prop}[\cite{DMC1} Prop. 2.2]\label{prop_GtoGHS}
Let $n \in \mathbb{Z}^{+}$ and, for $1 \leq i \leq n+1$, let $\frac{m_i}{d_i} \in \mathbb{Q}$ such that $0<\frac{m_i}{d_i}<1$. 
Let $p \equiv 1 \pmod {d_i}$ be prime and let $\rho_i$ be the character of order $d_i$ of $\mathbb{F}_p^*$ given by $\bar{\omega}^{\frac{p-1}{d_i}}$, where $\omega$ is the Teichm\"{u}ller character. Then
\begin{align*}
{_{n+1}G} \left( \tfrac{m_1}{d_1}, \tfrac{m_2}{d_2}, \dotsc, \tfrac{m_{n+1}}{d_{n+1}} \right)_p
=(-p)^n \,  {_{n+1}F_n}  {\Bigg( \begin{array}{cccc} \rho_1^{m_1}, & \rho_2^{m_2}, & \dotsc, & \rho_{n+1}^{m_{n+1}} \\
\phantom{\rho_1^{m_1}} & \varepsilon, & \dotsc, & \varepsilon \end{array}
\bigg| \; 1 \Biggr)}_{p}.
\end{align*}
\end{prop}

\noindent Using Proposition \ref{prop_GtoGHS} it is easy to see the following corollary to Theorem \ref{cor_GStoMod}.
\begin{cor}\label{Cor_GHStoMod}
If $p \equiv 1 \pmod 5$ is prime 
and $c(p)$ is as defined in (\ref{for_ModForm}), then 
\begin{equation*}
-p^3 \; _4F_3 \Biggl( \begin{array}{cccc} \chi_5, & \chi_5^2, & \chi_5^3, & \chi_5^4 \\
\phantom{\chi_5} & \varepsilon, & \varepsilon, & \varepsilon \end{array}
\bigg| \; 1 \Biggr)_p 
- p
= c(p).
\end{equation*}
\end{cor}

The remainder of the paper is organized as follows. Section 2 recalls some properties of Gauss and Jacobi sums, and the $p$-adic gamma function. The proof of Theorems \ref{thm_DMCMain} and \ref{cor_GStoMod} appear in Section 3.


\section{Preliminaries}
We briefly recall some properties of Gauss and Jacobi sums and the $p$-adic gamma function, and also develop some preliminary results which we will use in Section 3. Throughout we let $\mathbb{Z}_p$, $\mathbb{Q}_p$ and $\mathbb{C}_p$ denote the ring of $p$-adic integers, the field of $p$-adic numbers and the $p$-adic completion of the algebraic closure of $\mathbb{Q}_p$, respectively.
\subsection{Gauss and Jacobi Sums}
We first recall some properties of multiplicative characters. 
In particular, we note the following orthogonal relations. For $\chi \in \widehat{\mathbb{F}_p^{*}}$ we have
\begin{equation}\label{for_TOrthEl}
\sum_{x \in \mathbb{F}_p} \chi(x)=
\begin{cases}
p-1 & \text{if $\chi = \varepsilon$}  ,\\
0 & \text{if $\chi \neq \varepsilon$}  ,
\end{cases}
\end{equation}
and, for $x \in \mathbb{F}_p$ we have
\begin{equation}\label{for_TOrthCh}
\sum_{\chi \in \widehat{\mathbb{F}_p^{*}}} \chi(x)=
\begin{cases}
p-1 & \text{if $x=1$}  ,\\
0 & \text{if $x \neq 1$}  .
\end{cases}
\end{equation}

We now introduce some properties Gauss and Jacobi sums. For further details see \cite{BEW} and \cite{IR}, noting that we have adjusted results to take into account $\varepsilon(0)=0$. Let $\zeta_p$ be a fixed primitive $p$-th root of unity in $\bar{\mathbb{Q}_p}$. We define the additive character $\theta : \mathbb{F}_p \rightarrow \mathbb{Q}_p(\zeta_p)$ by $\theta(x):=\zeta_p^{x}$.
\noindent 
We note that $\mathbb{Q}_p$ contains all $(p-1)$-th roots of unity and in fact they are all in $\mathbb{Z}^{*}_p$. Thus we can consider multiplicative characters of $\mathbb{F}_p^{*}$ to be maps $\chi: \mathbb{F}_p^{*} \to \mathbb{Z}_{p}^{*}$. Recall then that for $\chi \in \widehat{\mathbb{F}_p^{*}}$, the Gauss sum $g(\chi)$ is defined by 
$g(\chi):= \sum_{x \in \mathbb{F}_p} \chi(x) \theta(x).$
It easily follows from (\ref{for_TOrthCh}) that we can express the additive character as a sum of Gauss sums. Specifically, for $x \in \mathbb{F}_p^*$\hspace{1pt} we have
\begin{equation}\label{for_AddtoGauss}
\theta(x)= \frac{1}{p-1}\sum_{\chi \in \widehat{\mathbb{F}_p^{*}}} g(\bar{\chi})  \, \chi(x).
\end{equation}
The following important result gives a simple expression for the product of two Gauss sums. For $\chi \in \widehat{\mathbb{F}_p^{*}}$ we have
\begin{equation}\label{for_GaussConj}
g(\chi)g(\bar{\chi})=
\begin{cases}
\chi(-1) p & \text{if } \chi \neq \varepsilon,\\
1 & \text{if } \chi= \varepsilon.
\end{cases}
\end{equation}
\noindent Another important product formula for Gauss sums is the Hasse-Davenport formula.
\begin{theorem}[Hasse, Davenport \cite{BEW} Thm 11.3.5]\label{thm_HD}
Let $\chi$ be a character of order $m$ of $\mathbb{F}_p^*$ for some positive integer $m$. For a character $\psi$ of $\mathbb{F}_p^*$ we have
\begin{equation*}
\prod_{i=0}^{m-1} g(\chi^i \psi) = g(\psi^m) \psi^{-m}(m)\prod_{i=1}^{m-1} g(\chi^i).
\end{equation*}
\end{theorem}

We now introduce generalized Jacobi sums.
Let $\chi_1, \chi_2, \dotsc, \chi_k \in \widehat{\mathbb{F}_p^{*}}$. Then the generalized Jacobi sum $J(\chi_1, \chi_2, \dotsc, \chi_k)$ of order $k$ is defined by
\begin{equation*}\label{for_GenJacSum}
J(\chi_1, \chi_2, \dotsc, \chi_k):= \sum_{\substack{t_1+t_2+ \dotsm + t_k=1\\ t_i \in \mathbb{F}_p}} \chi_1(t_1) \chi_2(t_2) \dotsm \chi_k(t_k).
\end{equation*}
There is a general formula relating generalized Jacobi sums of order $k$ to ones of order $k-1$.  One special case is the following. If $\chi_1\chi_2\dotsm\chi_k$ is trivial but at least one of $\chi_1, \chi_2, \dotsc, \chi_k$ is non-trivial, then
\begin{equation}\label{for_JacRedSmall}
J(\chi_1, \chi_2, \dotsc, \chi_k)=-\chi_k(-1)  J(\chi_1, \chi_2, \dotsc, \chi_{k-1}) \; .
\end{equation}

\noindent We can relate generalized Jacobi sums to Gauss sums in the following way.
For $\chi_1, \chi_2, \dotsc, \chi_k \in \widehat{\mathbb{F}_p^{*}}$ not all trivial,\\
\begin{equation}\label{for_JactoGauss}
J(\chi_1, \chi_2, \dotsc, \chi_k)=
\begin{cases}
\dfrac{g(\chi_1)g(\chi_2)\dotsc g(\chi_k)}{g(\chi_1 \chi_2 \dotsm \chi_k)}
& \qquad \text{if } \chi_1 \chi_2 \dotsm \chi_k \neq \varepsilon,\\[18pt]
-\dfrac{g(\chi_1)g(\chi_2)\dotsc g(\chi_k)}{p}
&\qquad \text{if }\chi_1 \chi_2 \dotsm \chi_k = \varepsilon \: .
\end{cases}
\end{equation}

\noindent We now develop some new results which we will use in Section 3.
\begin{lemma}\label{lem_JacTsum}
Let $p\equiv1 \pmod5$ be prime and let $\psi$ be a character of order $5$ of $\mathbb{F}_p^*$. If  $a,b,c \in \mathbb{Z}$ are such that $a+c$, $b+c \not\equiv 0 \pmod5$, then
\begin{equation*}
\sum_{\chi \in \widehat{\mathbb{F}_p^{*}}}  \chi(-1) J(\bar{\chi} \psi^a,\, \bar{\chi} \psi^b, \, \chi \psi^c)=-(p-1).
\end{equation*}
\end{lemma}
\begin{proof}
\begin{align*}
\sum_{\chi \in \widehat{\mathbb{F}_p^{*}}}  \chi(-1) J(\bar{\chi} \psi^a,\, \bar{\chi} \psi^b, \, \chi \psi^c)
&=\sum_{\chi \in \widehat{\mathbb{F}_p^{*}}}  \chi(-1)  \sum_{\substack{t_1+t_2+ t_3=1\\ t_i \in \mathbb{F}_p^*}} \bar{\chi} \psi^a(t_1) \, \bar{\chi} \psi^b(t_2) \, \chi \psi^c(t_3)\\
&=\sum_{\chi \in \widehat{\mathbb{F}_p^{*}}}  \chi(-1)   \sum_{\substack{t_1+t_2+ t_3=1\\ t_i \in \mathbb{F}_p^*}}\bar{\chi}\left(\frac{t_1t_2}{t_3}\right) \psi(t_1^a \; t_2^b \; t_3^c)\\
&= (p-1) \sum_{\substack{t_1+t_2+ t_3=1\\ t_i \in \mathbb{F}_p^*\\-\frac{t_1t_2}{t_3}=1}} \psi(t_1^a \; t_2^b \; t_3^c) \qquad  \text{by (\ref{for_TOrthCh}).}
\end{align*}
All possible triples $(t_1,t_2,t_3)$ satisfying the conditions of the summation can be represented by $(t_1,1,-t_1)$ and $(1,t_2,-t_2)$, not counting $(1,1,-1)$ twice. Hence
\begin{align*}
\sum_{\chi \in \widehat{\mathbb{F}_p^{*}}}  \chi(-1) & J(\bar{\chi} \psi^a,\, \bar{\chi} \psi^b, \, \chi \psi^c)\\
&=
(p-1) \left[\sum_{t_1\in \mathbb{F}_p^*} \psi(t_1^{a+c}\: (-1)^c) + \sum_{t_2\in \mathbb{F}_p^*} \psi(t_2^{b+c} \:(-1)^c) - \psi((-1)^c) \right]\\
&=(p-1)\;\psi^c(-1) \left[\sum_{t_1\in \mathbb{F}_p^*} \psi^{a+c}(t_1) + \sum_{t_2\in \mathbb{F}_p^*} \psi^{b+c}(t_2) - 1 \right].
\end{align*}
Now $a+c$,  $b+c \not\equiv 0 \pmod5$ so both $\psi^{a+c}$ and $\psi^{b+c}$ are non-trivial. Thus, by (\ref{for_TOrthEl}), 
$$\displaystyle \sum_{t_1\in \mathbb{F}_p^*} \psi^{a+c}(t_1) = \sum_{t_2\in \mathbb{F}_p^*} \psi^{b+c}(t_2)=0.$$
Also, $\psi(-1)=\psi^5(-1) =1,$ which completes the proof. 
\end{proof}

\begin{cor}\label{cor_GaussTsum}
Let $p\equiv1 \pmod5$ be prime and let $\psi$ be a character of order $5$ of $\mathbb{F}_p^*$. If  $a,b,c \in \mathbb{Z}$ are such that $a+c$, $b+c \not\equiv 0 \pmod5$, then
\begin{equation*}
\sum_{\chi \in \widehat{\mathbb{F}_p^{*}}}  g(\bar{\chi} \psi^a) \,g(\bar{\chi} \psi^b) \, g(\chi \psi^c) \, g(\chi \bar{\psi}{}^{a+b+c})=-p(p-1).
\end{equation*}
\end{cor}
\begin{proof}
Using (\ref{for_JacRedSmall}) and (\ref{for_JactoGauss}) we see that
\begin{align*}
\sum_{\chi \in \widehat{\mathbb{F}_p^{*}}}  g(\bar{\chi} \psi^a) \,g(\bar{\chi} \psi^b) \, & g(\chi \psi^c) \,  g(\chi \bar{\psi}{}^{a+b+c})\\
&=-p\sum_{\chi \in \widehat{\mathbb{F}_p^{*}}} J\left(\bar{\chi} \psi^a, \, \bar{\chi} \psi^b, \, \chi \psi^c, \, \chi \bar{\psi}{}^{a+b+c}\right) \\
&=p\sum_{\chi \in \widehat{\mathbb{F}_p^{*}}} \chi(-1) \; J\left(\bar{\chi} \psi^a, \, \bar{\chi} \psi^b, \, \chi \psi^c \right). 
\end{align*}
Applying Lemma \ref{lem_JacTsum} then yields the result.
\end{proof}

We now recall a formula for counting zeros of polynomials in affine space using the additive character.
If $f(x_1, x_2, \ldots x_n) \in \mathbb{F}_p[x_1, x_2, \ldots x_n]$, then the number of points, $N_p^*$, in $\mathbb{A}^n(\mathbb{F}_p)$ satisfying
$f(x_1, x_2, \ldots x_n) =0$ is given by
\begin{equation}\label{for_CtgPts}
p N_p^* = p^n +\sum_{y \in \mathbb{F}_p^*} \sum_{x_1, x_2, \ldots x_n \in \mathbb{F}_p}
\theta(y \: f(x_1, x_2, \ldots x_n)) \; .
\end{equation}


In \cite{Ko4}, Koblitz provides a formula for the number of points in $\mathbb{P}^{n-1}(\mathbb{F}_p)$ on the hypersurface 
$x_1^d + x_2^d + \dots + x_n^d - d \lambda x_1 x_2 \dots x_n=0$,
for some $d, \lambda \in \mathbb{F}_p$, where $p \equiv 1 \pmod d$. Let 
\begin{equation}\label{Def_W}
W:=\{w=(w_1, w_2, \ldots, w_n) \in \mathbb{Z}^n : 0 \leq w_i < d, \sum_{i=1}^n w_i \equiv 0 \pmod d\}. 
\end{equation}
\noindent Let $T$ be a fixed generator for the group of characters of $\mathbb{F}_p^*$ and set $t:=\frac{p-1}{d}$. 
Define an equivalence relation $\sim$ on $W$ by 
\begin{equation}\label{Def_tilde}
w \sim w \prime \textup{ if } w- w\prime \textup{ is a multiple modulo $d$ of $(1,1, \ldots, 1)$}. 
\end{equation}
Define 
\begin{equation}\label{Def_Np0}
N_p(0,w):=
\begin{cases}
0 & \text{if some but not all } w_i=0,\\[6pt]
\frac{p^{n-1}-1}{p-1} & \text{if all } w_i=0,\\[6pt]
\frac{1}{p} \prod_{i=1}^{n} g(T^{w_i t}) & \text{if all } w_i \neq 0.
\end{cases}
\end{equation}

\noindent Then we have the following theorem.
\begin{theorem}[Koblitz \cite{Ko4} Thm. 2]\label{thm_Koblitz}
Let $N_p(\lambda)$ be the number of points in $\mathbb{P}^{n-1}(\mathbb{F}_p)$ on $\sum_{i=1}^n x_i^d - d \lambda \prod_{j=1}^{n} x_i=0$, for some $d, \lambda \in \mathbb{F}_p$. Let $W$, $\sim$ and $N_p(0,w)$ be defined by (\ref{Def_W}), (\ref{Def_tilde}) and (\ref{Def_Np0}) respectively. Let $T$ be a fixed generator for the group of characters of $\mathbb{F}_p^*$, let $p \equiv 1 \pmod d$ and define $t:=\frac{p-1}{d}$. Then
\begin{equation*}
N_p(\lambda) = \sum_{w \in W} N_p(0,w) + \frac{1}{p-1} \sum_{[w] \in W/ \sim} \; \sum_{j=0}^{p-2} 
\frac{\prod_{i=1}^{n} g(T^{j+w_i t})}{g(T^{dj})}\; T^{dj}(d \lambda).\\
\end{equation*}
\end{theorem}
\vspace{3pt}

\subsection{$p$-adic preliminaries}
We first define the Teichm\"{u}ller character to be the primitive character $\omega: \mathbb{F}_p \rightarrow\mathbb{Z}^{*}_p$ satisfying $\omega(x) \equiv x \pmod p$ for all $x \in \{0,1, \ldots, p-1\}$.
We now recall the $p$-adic gamma function. For further details, see \cite{Ko}.
Let $p$ be an odd prime.  For $n \in \mathbb{Z}^{+}$ we define the $p$-adic gamma function as
\begin{align*}
\gfp{n} &:= {(-1)}^n \prod_{\substack{0<j<n\\p \nmid j}} j \\
\intertext{and extend to all $x \in\mathbb{Z}_p$ by setting $\gfp{0}:=1$ and} 
\gfp{x} &:= \lim_{n \rightarrow x} \gfp{n}
\end{align*}
for $x\neq 0$, where $n$ runs through any sequence of positive integers $p$-adically approaching $x$. 
This limit exists, is independent of how $n$ approaches $x$, and determines a continuous function
on $\mathbb{Z}_p$ with values in $\mathbb{Z}^{*}_p$.
We now state a product formula for the $p$-adic gamma function.
If $m\in\mathbb{Z}^{+}$, $p \nmid m$ and $x=\frac{r}{p-1}$ with $0\leq r \leq p-1$ then
\begin{equation}\label{for_pGammaMult}
\prod_{h=0}^{m-1} \gfp{\tfrac{x+h}{m}}=\omega\left(m^{(1-x)(1-p)}\right)
\gfp{x} \prod_{h=1}^{m-1} \gfp{\tfrac{h}{m}}.
\end{equation}

\noindent The Gross-Koblitz formula \cite{GK} allows us to relate Gauss sums and the $p$-adic gamma function. Let $\pi \in \mathbb{C}_p$ be the fixed root of $x^{p-1}+p=0$ which satisfies ${\pi \equiv \zeta_p-1 \pmod{{(\zeta_p-1)}^2}}$. Then we have the following result.
\begin{theorem}[Gross, Koblitz \cite{GK}]\label{thm_GrossKoblitz}
For $0 \leq j \leq p-2$,
\begin{equation*} 
g(\bar{\omega}^j)=-\pi^j \: \gfp{\tfrac{j}{p-1}}.
\end{equation*}
\end{theorem}

\vspace{3pt}

\section{Proofs}
\begin{proof}[Proof of Theorem \ref{cor_GStoMod}]
Let $N_p$ be the number of points in $\mathbb{P}^4(\mathbb{F}_p)$ on $x_1^5+x_2^5+x_3^5+x_4^5+ x_5^5-5 x_1 x_2 x_3 x_4 x_5=0$. Then from \cite[page 32]{Me} (following the work of Schoen \cite{S}), we have
\begin{equation*}\label{for_ModtoPts}
c(p)=
\begin{cases}
p^3 + 25 p^2 -100 p + 1 - N_p & \text{if $p \equiv 1\phantom{,2} \pmod {5}$},\\
p^3 + \phantom{25} p^2 \: \;\phantom{-100 p} + 1 - N_p & \text{if $p \equiv 4\phantom{,2} \pmod {5}$}, \\
p^3 + \phantom{25} p^2 + \phantom{10}2 p + 1 - N_p & \text{if $p \equiv 2,3 \pmod {5}$}.
\end{cases}
\end{equation*}
Therefore, noting that
\begin{equation*}\label{for_Sp}
s(p)=
\begin{cases}
+1 & \text{if } p\equiv 1, 4 \pmod 5,\\
-1 & \text{if } p\equiv 2,3 \pmod 5,
\end{cases}
\end{equation*}
it suffices to prove
\begin{equation}\label{N_to_G}
N_p = - {_{4}G} \left(\tfrac{1}{5} ,\hspace{2pt} \tfrac{2}{5},\hspace{2pt} \tfrac{3}{5} ,\hspace{2pt} \tfrac{4}{5}\right)_p +
\begin{cases}
p^3 + 25 p^2 -99 p + 1 & \text{if $p \equiv 1 \pmod {5}$},\\
p^3 + \phantom{25} p^2 + \phantom{99} p + 1 & \text{if $p \not\equiv 1 \pmod {5}$}.
\end{cases}
\end{equation}
We will express both sides of (\ref{N_to_G}) in terms of Gauss sums, starting with ${_{4}G} \left(\tfrac{1}{5} ,\hspace{2pt} \tfrac{2}{5},\hspace{2pt} \tfrac{3}{5} ,\hspace{2pt} \tfrac{4}{5}\right)_p$.

Let $m_0:=-1$, $m_{5}:=p-2$, $d_0=d_{5}:=p-1$ and $m_i:=i$, $d_i:=5$ for $1\leq i \leq4$. Also, let $r_i=\frac{p-1}{d_i}$ for $0 \leq i \leq 5$. Consequently, from definition (\ref{for_GFn}) we get that 
\begin{multline*}
{_{4}G} \left(\tfrac{1}{5} ,\hspace{2pt} \tfrac{2}{5},\hspace{2pt} \tfrac{3}{5} ,\hspace{2pt} \tfrac{4}{5}\right)_p\\
= \frac{-1}{p-1} \sum_{k=0}^{4} (-p)^k \sum_{j=\left\lfloor m_k r_k \right\rfloor +1}^{\left\lfloor m_{k+1}r_{k+1}\right\rfloor} 
\frac{{\biggfp{\tfrac{j}{p-1}}}^{4}}{\biggfp{1-\tfrac{j}{p-1}}}
\frac
{\prod_{h=0}^{4} \Biggfp{\frac{k+1-\frac{5j}{p-1}+h}{5}}}
{\prod_{h=1}^{4} \biggfp{\frac{h}{5}}}.
\end{multline*}
For a given $k$, we easily see that $k\bigl(\frac{p-1}{5}\bigr) \leq j \leq (k+1)\bigl(\frac{p-1}{5}\bigr)$, with equality on the left when $k=j=0$. Therefore $0 \leq (k+1)(p-1) -5j \leq p-1$ and we can apply (\ref{for_pGammaMult}) with $m=5$ and $x= k+1-\frac{5j}{p-1}$ to get
\begin{multline*}
{_{4}G} \left(\tfrac{1}{5} ,\hspace{2pt} \tfrac{2}{5},\hspace{2pt} \tfrac{3}{5} ,\hspace{2pt} \tfrac{4}{5}\right)_p\\
= \frac{-1}{p-1} \sum_{k=0}^{4} (-p)^k \sum_{j=\left\lfloor m_k r_k \right\rfloor +1}^{\left\lfloor m_{k+1}r_{k+1}\right\rfloor} 
\frac{{\biggfp{\tfrac{j}{p-1}}}^{4}}{\biggfp{1-\tfrac{j}{p-1}}}
\biggfp{k+1-\tfrac{5j}{p-1}}
\omega(5^{-5j}).
\end{multline*}
We now use Theorem \ref{thm_GrossKoblitz} to convert this to an expression involving Gauss sums. To satisfy the conditions of the theorem for all arguments of the $p$-adic gamma functions above, we split off the term when $j=0$. For all other values of $j$, we have $0 \leq (p-1)- j, (k+1)(p-1)-5j \leq p-2$. We then get
\begin{equation*}
{_{4}G} \left(\tfrac{1}{5} ,\hspace{2pt} \tfrac{2}{5},\hspace{2pt} \tfrac{3}{5} ,\hspace{2pt} \tfrac{4}{5}\right)_p
= \frac{-1}{p-1} \left[1+ \sum_{j=1}^{p-2} \frac{g(\omega^{-j})^4}{g(\omega^j)} \hspace{2pt} g(\omega^{5j})\hspace{2pt}  \omega(5^{-5j}) \right].
\end{equation*}
Now applying (\ref{for_GaussConj}) yields
\begin{equation*}
{_{4}G} \left(\tfrac{1}{5} ,\hspace{2pt} \tfrac{2}{5},\hspace{2pt} \tfrac{3}{5} ,\hspace{2pt} \tfrac{4}{5}\right)_p
= -\frac{1}{p-1} \left[1+\frac{1}{p} \sum_{j=1}^{p-2}\; g(\omega^{-j})^{5}\, g(\omega^{5j}) \; \omega^{-5j}(-5) \right].
\end{equation*}

We will now evaluate $N_p$. We first consider the case when $p \equiv 1 \pmod 5$. By Theorem \ref{thm_Koblitz} with $d=5, \lambda=1$ we see that
\begin{equation}\label{for_1mod5}
N_p = \sum_{w \in W} N_p(0,w) + \frac{1}{p-1} \sum_{[w] \in W/ \sim} \sum_{j=0}^{p-2} 
\frac{\prod_{i=1}^{5} g(T^{j+w_i t})}{g(T^{5j})}\; T^{5j}(5),
\end{equation}
where $W$, $\sim$ and $N_p(0,w)$ are as defined in (\ref{Def_W}), (\ref{Def_tilde}) and (\ref{Def_Np0}) respectively, $T$ is a fixed generator for the group of characters of $\mathbb{F}_p$ and $t=\frac{p-1}{5}$. We now describe the elements of both $W$ and $W/ \sim$ as they apply to our setting. We first note that the contribution of any $(w_1, w_2, \ldots, w_5)$ to the above formula is the same as any permutation of it. We therefore list the elements of these sets up to permutation. We will however indicate, using a superscript, the multiplicity with which it contributes to the formula, i.e., its total number of distinct permutations. As $N_p(0,w)$ vanishes if some but not all of the $w_i$ are zero we will just list the elements of $W$ for which all $w_i$ are non-zero and call it $W^*$. Therefore
\begin{align*}
W^* = \{&(1,1,1,1,1)^1, (2,2,2,2,2)^1, (3,3,3,3,3)^1, (4,4,4,4,4)^1,\\
& (1,1,1,3,4)^{20}, (1,2,2,2,3)^{20}, (2,3,3,3,4)^{20}, (1,2,4,4,4)^{20},\\
& (1,1,2,2,4)^{30}, (2,2,3,3,4)^{30},(1,1,2,3,3)^{30}, (1,3,3,4,4)^{30} \} 
\end{align*}
and
\begin{multline*}
W/\sim \; = \{(0,0,0,0,0)^1, (0,1,2,3,4)^{24}, (0,0,0,1,4)^{20},\\ (0,0,0,2,3)^{20}, (0,0,1,1,3)^{30}, (0,0,1,2,2)^{30} \}.
\end{multline*}
We now use $W^*$ to evaluate $N_p(0):=\sum_{w \in W} N_p(0,w)$. We note that many of the elements of $W^*$ are multiples modulo $5$ of each other. Thus
\begin{multline*}
N_p(0) = \frac{p^4 - 1}{p-1} + \frac{20}{p} \sum_{i=1}^{4} g(T^{it})^3 \; g(T^{3it}) \; g(T^{4it})\\ + \frac{30}{p} \sum_{i=1}^{4} g(T^{it})^2 \; g(T^{2it})^2 \; g(T^{4it}) + \frac{1}{p} \sum_{i=1}^{4} g(T^{it})^5.
\end{multline*}
Applying (\ref{for_GaussConj}) then yields
\begin{align}\label{for_Np0}
N_p(0)
&= p^3 + p^2 + p + 1 + 50 \sum_{i=1}^{4} g(T^{it})^2 \; g(T^{3it}) + \frac{1}{p} \sum_{i=1}^{4} g(T^{it})^5.
\end{align}
We now focus on the second sum on the right-hand side of (\ref{for_1mod5}) and evaluate it for each individual $[w] \in W/\sim$ (up to permutation). We will denote each such minor sum as $S_{[w]}$. We start with $[w]= (0,0,0,0,0).$ So
\begin{equation*}
S_{(0,0,0,0,0)} = \frac{1}{p-1} \sum_{j=0}^{p-2} \frac{g(T^{j})^5}{g(T^{5j})} \; T^{5j}(5).
\end{equation*}
We isolate the terms where $j$ is a multiple of $t= \frac{p-1}{5}$ and apply (\ref{for_GaussConj}) to the other terms to get
\begin{equation}\label{for_S00000}
S_{(0,0,0,0,0)} 
= \frac{1}{p-1} \left [1 - \sum_{\substack{j=1\\t \mid j}}^{p-2} g(T^{j})^5 + \frac{1}{p} \sum_{\substack{j=1\\t \nmid j}}^{p-2} g(T^{j})^5 \; g(T^{-5j})\; T^{5j}(-5)
 \right].
\end{equation}
Next considering $[w]=(0,1,2,3,4)$ yields
\begin{equation*}
S_{(0,1,2,3,4)} 
=\frac{24}{p-1} \sum_{j=0}^{p-2} \frac{g(T^j) \; g(T^{j+t}) \; g(T^{j+2t}) \; g(T^{j+3t}) \; g(T^{j+4t})}{g(T^{5j})} \; T^{5j}(5).
\end{equation*}
Combining Theorem \ref{thm_HD}, with $m=5$ and $\psi=T^j$, and (\ref{for_GaussConj}) we see that
\begin{equation}\label{for_HD}
g(T^j) \; g(T^{j+t})\; g(T^{j+2t})\; g(T^{j+3t})\; g(T^{j+4t}) = g(T^{5j})\; T^{-5j}(5)\; p^2
\end{equation}
Therefore, we have that
\begin{equation}\label{for_S01234}
S_{(0,1,2,3,4)} 
 = 24 \, p^2.
\end{equation}
The sums for the remaining elements, $(0,0,0,1,4)$, $(0,0,0,2,3)$, $(0,0,1,1,3)$  and $(0,0,1,2,2)$, can all be evaluated in a similar manner to each other. By definition 
\begin{equation*}
S_{(0,0,0,1,4)} 
=\frac{20}{p-1} \sum_{j=0}^{p-2} \frac{g(T^j)^3 \; g(T^{j+t}) \; g(T^{j+4t})}{g(T^{5j})} \; T^{5j}(5).
\end{equation*}
Applying (\ref{for_HD}) gives us
\begin{equation*}
S_{(0,0,0,1,4)} 
=\frac{20 \, p^2}{p-1} \sum_{j=0}^{p-2} \frac{g(T^j)^2}{g(T^{j+2t})\; g(T^{j+3t})}.
\end{equation*}
We now use (\ref{for_GaussConj}) to clear denominators, being careful to deal with the cases $j=2t,3t$ separately. Thus 
\begin{multline*}
S_{(0,0,0,1,4)} 
=\frac{20}{p-1} \left[\; \sum_{\substack{j=0\\j \neq 2t, 3t}}^{p-2} g(T^j)^2 \; g(T^{-j+3t})\; g(T^{-j+2t})
\right. \\ \left.
 \phantom{\sum_{\substack{j=0\\j \neq 2t, 3t}}^{p-2}} - p \, g(T^{2t})^2 \; g(T^t) - p \, g(T^{3t})^2 \; g(T^{4t}) \right].
\end{multline*}
We now apply Corollary \ref{cor_GaussTsum} with $a=3$, $b=2$ and $c=0$ to get
\begin{equation}\label{for_S00014}
S_{(0,0,0,1,4)}
=-20p - 20 \left[g(T^{2t})^2 \; g(T^t) + g(T^{3t})^2 \; g(T^{4t}) \right].
\end{equation}

\newpage
\noindent Similarly, we have 
\begin{equation}\label{for_S00023}
S_{(0,0,0,2,3)}=-20p - 20 \left[g(T^{t})^2 \; g(T^{3t}) + g(T^{4t})^2 \; g(T^{2t}) \right],
\end{equation}
\begin{equation}\label{for_S00113}
S_{(0,0,1,1,3)}=-30p - 30 \left[g(T^{3t})^2 \; g(T^{4t}) + g(T^{2t})^2 \; g(T^{t}) \right]
\end{equation}
and\\[-6pt]
\begin{equation}\label{for_S00122}
S_{(0,0,1,2,2)}=-30p - 30 \left[g(T^{t})^2 \; g(T^{3t}) + g(T^{4t})^2 \; g(T^{2t}) \right].\\[9pt]
\end{equation}

\noindent Combining (\ref{for_Np0}), (\ref{for_S00000}), (\ref{for_S01234}), (\ref{for_S00014}), (\ref{for_S00023}), (\ref{for_S00113}) and (\ref{for_S00122}) we get via (\ref{for_1mod5}) that
\begin{equation*}
N_p = p^3 + 25 p^2 -99 p + 1 + \frac{1}{p-1} \left [1 + \frac{1}{p} \sum_{j=1}^{p-2} g(T^{j})^5 \; g(T^{-5j})\; T^{5j}(-5)
 \right].
\end{equation*}
Taking $T$ to be $\omega^{-j}$ gives us
$$N_p = p^3 + 25 p^2 -99 p + 1 - {_{4}G} \left(\tfrac{1}{5} ,\hspace{2pt} \tfrac{2}{5},\hspace{2pt} \tfrac{3}{5} ,\hspace{2pt} \tfrac{4}{5}\right)_p$$
as required.

We now evaluate $N_p$ when $p \not\equiv 1 \pmod 5$, using (\ref{for_CtgPts}), and express our results in terms of Gauss sums using (\ref{for_AddtoGauss}). Let $\bar{x}$ denote the tuple $(x_1, x_2, x_3, x_4, x_5)$ and define $f(\bar{x}) := x_1^5+x_2^5+x_3^5+x_4^5+ x_5^5-5 x_1 x_2 x_3 x_4 x_5$ for brevity. Also, let $N_p^A$ denote the number of points in $\mathbb{A}^5(\mathbb{F}_p)$ on $f(\bar{x})=0.$ Then, as $f$ is homogeneous,
\begin{equation}\label{for_Np}
N_p = \frac{N_p^A-1}{p-1}.
\end{equation}
From (\ref{for_CtgPts}) we see that
\begin{equation*}
p N_p^A 
= p^5 + \sum_{y \in \mathbb{F}_p^*} \sum_{\substack{x_i \in \mathbb{F}_p\\some \; x_i=0}} \theta(y \: f(\bar{x})) + \sum_{y \in \mathbb{F}_p^*} \sum_{x_i \in \mathbb{F}_p^{*}} \theta(y \: f(\bar{x})).
\end{equation*}
We now consider the number of points $N_p^{\prime}$ in $\mathbb{A}^5(\mathbb{F}_p)$ on $f^{\prime}(\bar{x}):=x_1^5+x_2^5+x_3^5+x_4^5+ x_5^5=0.$ As $x \rightarrow x^5$ is an automorphism on $\mathbb{F}_p$ when $p \not\equiv 1 \pmod 5$, it is easy to see that $N_p^{\prime}=p^4$. From (\ref{for_CtgPts}) we also have that
\begin{align*}
p N_p^{\prime} 
&= p^5 + \sum_{y \in \mathbb{F}_p^*} \sum_{\substack{x_i \in \mathbb{F}_p\\some \; x_i=0}} \theta(y \: f^{\prime}(\bar{x})) + \sum_{y \in \mathbb{F}_p^*} \sum_{x_i \in \mathbb{F}_p^{*}} \theta(y \: f^{\prime}(\bar{x})).
\end{align*}
Therefore
$$\sum_{y \in \mathbb{F}_p^*} \sum_{\substack{x_i \in \mathbb{F}_p\\some \; x_i=0}} \theta(y \: f^{\prime}(\bar{x})) =-  \sum_{y \in \mathbb{F}_p^*} \sum_{x_i \in \mathbb{F}_p^{*}} \theta(y \: f^{\prime}(\bar{x})).$$
Noting that
$$\sum_{y \in \mathbb{F}_p^*} \sum_{\substack{x_i \in \mathbb{F}_p\\some \; x_i=0}} \theta(y \: f(\bar{x})) = 
\sum_{y \in \mathbb{F}_p^*} \sum_{\substack{x_i \in \mathbb{F}_p\\some \; x_i=0}} \theta(y \: f^{\prime}(\bar{x}))$$
yields
\begin{equation}\label{for_NpA}
p N_p^A = p^5 + \sum_{y \in \mathbb{F}_p^*} \sum_{x_i \in \mathbb{F}_p^{*}} \theta(y \: f(\bar{x})) 
-  \sum_{y \in \mathbb{F}_p^*} \sum_{x_i \in \mathbb{F}_p^{*}} \theta(y \: f^{\prime}(\bar{x})).
\end{equation}
We now convert the two sums on the right above, which we call $S_1$ and $S_2$ respectively, to expressions involving Gauss sums using (\ref{for_AddtoGauss}). Then, starting  with $S_2$, we have
\begin{align*}
S_2
&= \sum_{y, x_i \in \mathbb{F}_p^{*}} \theta(y \: x_1^5) \theta(y \: x_2^5)\theta(y \: x_3^5)\theta(y \: x_4^5)\theta(y \: x_5^5)\\
&= \frac{1}{(p-1)^5} \sum_{a,b,c,d,e = 0}^{p-2} \sum_{y, x_i \in \mathbb{F}_p^{*}}  g(T^{-a}) \; g(T^{-b}) \; g(T^{-c}) \; g(T^{-d}) \; g(T^{-e}) \; \\
&\qquad \qquad \qquad \qquad\;  \qquad \qquad
\cdot T^a(y x_1^5)\; T^b(y x_2^5)\; T^c(y x_3^5)\; T^d(y x_4^5)\; T^e(y x_5^5)\\
&=\frac{1}{(p-1)^5} \sum_{a,b,c,d,e = 0}^{p-2} \sum_{x_i \in \mathbb{F}_p^{*}}  g(T^{-a}) \; g(T^{-b}) \; g(T^{-c}) \; g(T^{-d}) \; g(T^{-e}) \; \\
& \qquad \qquad  \; \qquad
\cdot T^a(x_1^5)\;  T^b(x_2^5)\;  T^c(x_3^5)\; T^d(x_4^5)\;  T^e(x_5^5) \sum_{y \in \mathbb{F}_p^{*}} T^{a+b+c+d+e}(y).
\end{align*}
We now apply (\ref{for_TOrthEl}) to the last summation on the right, which yields $(p-1)$ if $e=-a-b-c-d$ and zero otherwise.
So  
\begin{align*}
S_2
&=\frac{1}{(p-1)^4} \sum_{a,b,c,d = 0}^{p-2} \sum_{x_i \in \mathbb{F}_p^{*}}  g(T^{-a}) \; g(T^{-b}) \; g(T^{-c}) \; g(T^{-d}) \; g(T^{a+b+c+d}) \; \\
&\qquad \qquad \qquad \quad \qquad \qquad
\cdot T^a(x_1^5)\; T^b(x_2^5)\; T^c(x_3^5)\; T^d(x_4^5)\; T^{-a-b-c-d}(x_5^5)\\
&=\frac{1}{(p-1)^4} \sum_{a,b,c,d = 0}^{p-2} \sum_{x_2, x_3, x_4, x_5 \in \mathbb{F}_p^{*}}  g(T^{-a}) \, g(T^{-b}) \, g(T^{-c}) \, g(T^{-d}) \, g(T^{a+b+c+d}) \; \\
&\qquad \qquad  \qquad \qquad
\cdot    T^b(x_2^5)\; T^c(x_3^5)\; T^d(x_4^5)\; T^{-a-b-c-d}(x_5^5)\sum_{x_1 \in \mathbb{F}_p^{*}} T^{5a}(x_1)\; .
\end{align*}
We again apply (\ref{for_TOrthEl}) to the last summation on the right
and continue in this manner isolating the sum for each $x_i$ in turn and applying (\ref{for_TOrthEl}). This leads to 
\begin{equation}\label{for_S2}
S_2
=- \sum_{x_5 \in \mathbb{F}_p^{*}} 1 = -(p-1). 
\end{equation}
A similar evaluation of $S_1$ using (\ref{for_AddtoGauss}) and (\ref{for_TOrthEl}) yields
\begin{equation}\label{for_S1}
S_1 = \sum_{e=0}^{p-2} g(T^{-e})^5 \; g(T^{5e}) T^{-5e}(-5). 
\end{equation}
Accounting for (\ref{for_S2}), (\ref{for_S1}) in (\ref{for_NpA}) we see that, via (\ref{for_Np}),
\begin{align*}
N_p 
&= \frac{1}{p-1} \left\{ \frac{1}{p} \left[ p^5 + \sum_{e=0}^{p-2} g(T^{-e})^5 \; g(T^{5e}) T^{-5e}(-5) +p -1 \right] - 1 \right\}\\
&= \frac{p^4-1}{p-1} + \frac{1}{p-1}\left[1 + \frac{1}{p} \sum_{e=1}^{p-2} g(T^{-e})^5 \; g(T^{5e}) T^{-5e}(-5) \right].
\end{align*}

\newpage
\noindent Taking $T$ to be $\omega^{j}$ gives us
$$N_p = p^3 + p^2 + p + 1 - {_{4}G} \left(\tfrac{1}{5} ,\hspace{2pt} \tfrac{2}{5},\hspace{2pt} \tfrac{3}{5} ,\hspace{2pt} \tfrac{4}{5}\right)_p$$
as required.
\end{proof}

\vspace{12pt}

\begin{proof}[Proof of Theorem \ref{thm_DMCMain}]
One easily checks the result for $p=2$. Combining Theorem \ref{thm_4G2} with $d=5$ and Theorem \ref{cor_GStoMod} yields the result.
\end{proof}

\section{Remark}
Using (\ref{for_CtgPts}) to count points on certain algebraic varieties is by no means new. This author first observed the technique in \cite{F2}. In the proof above, we have applied this technique in the case $p \not\equiv 1 \pmod 5$. We could have also applied this method to the case $p \equiv 1 \pmod 5$, choosing instead to use Theorem \ref{thm_Koblitz} for reasons of brevity. The two methods are essentially the same with Theorem \ref{thm_Koblitz} encapsulating much of the work which must be done if (\ref{for_CtgPts}) is used. For a detailed account of how the case $p \equiv 1 \pmod 5$ is established using (\ref{for_CtgPts}), please see \cite{DMC}.

\end{document}